\documentclass[12pt,a4paper]{amsart}
\usepackage[a4paper]{geometry}
\usepackage{mathptmx} 
\DeclareMathAlphabet{\mathcal}{OMS}{cmsy}{m}{n} 

\usepackage{perpage} 

\newtheorem{theorem}{Theorem}[section]
\newtheorem{lemma}[theorem]{Lemma}

\theoremstyle{definition}
\newtheorem{remark}[theorem]{Remark}

\newcommand{\vertbar}{\>|\>}
\newcommand{\set}[2]{\ensuremath{\{ #1 \vertbar #2 \}}}

\MakePerPage[2]{footnote}

\def\liebrack  {\ensuremath{[\,\cdot\, , \cdot\,]}} 

\DeclareMathOperator{\dcobound}{d}
\DeclareMathOperator{\End}{End}
\DeclareMathOperator{\id}{id}
\DeclareMathOperator{\im}{Im}
\DeclareMathOperator{\Ker}{Ker}
\DeclareMathOperator{\Z}{Z}

\begin{document}

\title{Yet another proof of the Ado theorem}
\author{Pasha Zusmanovich}
\address{
Department of Mathematics, University of Ostrava, Ostrava, Czech Republic
}
\email{pasha.zusmanovich@osu.cz}
\date{last revised July 29, 2018}
\thanks{J. Lie Theory \textbf{26} (2016), no.~3, 673--681; arXiv:1507.02233}
\thanks{
The financial support of the Regional Authority of the Moravian-Silesian Region
(grant MSK 44/3316) and of the Ministry of Education and Science of the 
Republic of Kazakhstan (grant 0828/GF4) is gratefully acknowledged}

\begin{abstract}
We give a simple proof of the Birkhoff theorem about existence of a faithful 
representation for any finite-dimensional nilpotent Lie algebra of 
characteristic zero.
\end{abstract}

\maketitle

\section{Introduction}

The Ado theorem says that any fi\-ni\-te-di\-men\-si\-o\-nal Lie algebra admits a faithful 
fi\-ni\-te-di\-men\-si\-o\-nal representation. It was first proved by Ado in 1935 
\cite{ado}\footnote[2]{
A bit of trivia: Sophus Lie had no doubt that (speaking in modern terms) every 
finite-dimensional Lie algebra admits a faithful finite-dimensional 
representation, but he was unable to prove this (cf. 
\cite[footnote at p.~598]{lie}). Nikolai Grigorievich Chebotarev has put his 
student Igor Dmitrievich Ado to the task. Ado presented his work as a Candidate
(a Russian equivalent of PhD) dissertation, but was awarded a Doctor degree (a 
Russian equivalent of Habilitation) instead, an extremely rare event in Russian
academic officialdom.
}
using a Lie-group technique (and as such, was restricted to the fields of 
complex and real numbers). Since then a few different proofs were given,
including purely algebraic ones and those valid in the positive characteristic
(the latter is due to Iwasawa; cf. \cite[\S 7, Exercice 5]{bourbaki}).
A relatively recent new proof, due to Neretin, is given in \cite{neretin}. All 
the known proofs involve the universal enveloping algebra, an 
infinite-dimensional object, and nontrivial facts about it (notably, the
Poincar\'e--Birkhoff--Witt theorem). Moreover, the arguments are different in 
zero and positive characteristics. Somewhat unusually, the positive 
characteristic case of the theorem is much easier than the characteristic zero 
one, due to the possibility to employ finite-dimensional induced modules 
constructed with the aid of a reduced finite-dimensional version of the 
universal enveloping algebra (in characteristic zero, such modules would be 
infinite). Perhaps because of all this, the Ado theorem is sometimes referred --
in writings, and also in talks and private conversations -- as a 
``strange theorem'' (\cite{neretin}) ``surprisingly tricky to prove'' 
(\cite[\S 2.3]{tao}).

Here we give an entirely different proof of the Ado theorem, basing on 
properties of free nilpotent Lie algebras, and on simple combinatorics related
to the tensor product of representations. The proof is elementary and does not 
involve universal enveloping algebras (in fact, it does not involve associative 
algebras at all and is intrinsic to the category of finite-dimensional Lie 
algebras). Another interesting feature of the proof is that it employs induction
on the dimension of the algebra -- not ``from below'', as in many existing 
proofs of the Ado theorem, but ``from above'', descending from an algebra for 
which the Ado theorem is already established.

The drawbacks of the proof are that it is valid for nilpotent algebras and in 
characteristic zero only (first established in full generality by Birkhoff in 
1937 \cite{birkhoff}, thus sometimes referred as ``the Birkhoff theorem''). 

We present the proof in Section \ref{sec-proof} as a series of (elementary) 
lemmas. The Ado theorem is one of the cornerstone results of today's structure theory of Lie algebras, used in
plenty of proofs and arguments. When pretending to give a new proof of such a 
basic result, one should be especially careful not to fall into the trap of 
circular arguments. That is why, even when using known and/or elementary 
results, we at least outline their proofs. We also carefully isolate places 
where we need assumptions such as nilpotency of the algebra and characteristic 
zero of the ground field.

Being different, our proof still shares with all the previous proofs some 
elementary (and seemingly unavoidable in this context) tricks and observations.
In particular, as for a centerless Lie algebra the Ado theorem is trivial 
(the adjoint representation will do), we will concentrate on central elements, 
and take the direct sum of representations to assemble ``local'' Ado properties
(nonvanishing on particular elements) into the ``global'' one (faithfulness).

In Section \ref{sec-alt} we speculate about possibility to extend this proof to the 
case of positive characteristic.

\section{The proof}\label{sec-proof}

In what follows, the ground field $K$ is arbitrary, and all algebras, modules 
and vector spaces are finite-dimensional, unless stated otherwise.

A Lie algebra structure on the tensor product $L \otimes A$ of a Lie algebra 
$L$ and an associative commutative algebra $A$ is defined by the obvious 
factor-wise multiplication: $[x \otimes a, y \otimes b] = [x,y] \otimes ab$ 
for $x,y \in L, a,b \in A$ (such Lie algebras are dubbed as 
\emph{current Lie algebras}, the term coming from physics where they play a 
role).

\begin{lemma}[\sc Embedding of graded algebras into a tensor product]\label{lemma-emb}
An $\mathbb N$-graded Lie algebra $L$ is embedded into the Lie algebra 
$L \otimes tK[t]/(t^n)$ for some $n \in \mathbb N$.
\end{lemma}

(We denote by $\mathbb N$ the set of all positive integers, so the algebras we 
consider are positively-graded).

\begin{proof}
Let $L = \bigoplus_{i=1}^{n-1} L_i$ be the $\mathbb N$-grading (necessarily 
finite, as $L$ is finite-dimensional). The algebra $L$ is embedded into 
$L \otimes tK[t]$ via $x \mapsto x \otimes t^i$, where 
$x\in L_i$, $1 \le i \le n-1$. Since $[L_i \otimes t^i, L_j \otimes t^j] = 0$ 
if $i+j \ge n$, this embedding factors through the ideal $L \otimes t^n K[t]$.
\end{proof}

\begin{remark}
The statement is obviously true for arbitrary, not necessary Lie, 
$\mathbb N$-graded algebras.
\end{remark}

For the purpose of this note, a representation $\rho$ of a Lie algebra $L$
(or the corresponding $L$-module) will be called \emph{nilpotent}, if $\rho(x)$
is a nilpotent linear map for any $x\in L$ (of course, due to the Engel theorem,
this is equivalent to existence of a positive integer $n$ such that 
$\rho(x_1) \cdots \rho(x_n) = 0$ for any $x_1, \dots, x_n \in L$, but we will not need that).

Recall that given a representation $\rho: L \to \End(V)$ of a Lie algebra $L$,
an \emph{$1$-cocycle} is a linear map $\varphi: L \to V$ such that
$$
\varphi([x,y]) - \rho(x)(\varphi(y)) + \rho(y)(\varphi(x)) = 0
$$
for any $x,y\in L$. The space of all such $1$-cocycles is denoted by 
$\Z^1(L,V)$, and they may be thought as derivations of $L$ with values in the 
$L$-module $V$.

\begin{lemma}[\sc A nondegenerate cocycle implies Ado]\label{lemma-coc}
Let $L$ be a Lie algebra, $V$ an $L$-module 
(respectively, nilpotent $L$-module), and $\varphi$ is $1$-cocycle in 
$\Z^1(L,V)$ such that $\Ker \varphi = 0$. Then $L$ has a faithful 
representation (respectively, faithful nilpotent representation).
\end{lemma}

\begin{proof}
The required representation $\rho$ is given by an action of $L$ on 
$V \oplus \Z^1(L,V)$ (direct sum of vector spaces), defined naturally on the 
first direct summand, and via $\rho(x) (\psi) = \psi(x)$ for $x\in L$ and
$\psi \in \Z^1(L,V)$, on the second direct summand.
\end{proof}

\begin{remark}
It is possible to extend this statement, via induction, to higher-order 
cocycles, but we will not need this (but see Remark \ref{rem-dual} below).
\end{remark}

\begin{lemma}[\sc Ado for graded algebras]\label{lem-loc-grad}
An $\mathbb N$-graded Lie algebra over a field of characteristic zero has a 
faithful nilpotent representation.
\end{lemma}

\begin{proof}
Let $L$ be such Lie algebra. By Lemma \ref{lemma-emb}, $L$ is embedded into 
$L \otimes tK[t]/(t^n)$. The derivation (i.e., $1$-cocycle with values in the 
adjoint module) $\id_L \otimes\, t\frac{\dcobound}{\dcobound t}$ of 
$L \otimes tK[t]/(t^n)$ acts on a nonzero element 
$\sum_{i=1}^{n-1} x_i \otimes t^i$ non-trivially, hence has zero kernel. By
Lemma \ref{lemma-coc}, $L \otimes tK[t]/(t^n)$ has a faithful nilpotent 
representation. Hence so does its subalgebra $L$.
\end{proof}

\begin{lemma}[\sc Ado for free nilpotent algebras]\label{lemma-free-nilp}
A free nilpotent Lie algebra of finite rank over a field of characteristic zero
has a faithful nilpotent representation.
\end{lemma}

\begin{proof}
A free nilpotent Lie algebra of finite rank is finite-dimensional and 
$\mathbb N$-graded. Apply Lemma \ref{lem-loc-grad}.
\end{proof}

\begin{lemma}[\sc Local Ado implies global Ado\footnote{
A reader of an earlier draft of this note remarked that after the present array
of lemmas, one might expect the next one to be called {\sc Much Ado about nothing}.
The lemmas are elementary indeed.
}]\label{lemma-local}
Let $L$ be a Lie algebra such that for any nonzero $x \in L$ there is a 
nilpotent representation $\rho_x$ of $L$ such that $\rho_x(x) \ne 0$. Then $L$ 
has a faithful nilpotent representation.
\end{lemma}

\begin{proof}
Pick a nonzero element $x_1\in L$, and set $\rho_1 = \rho_{x_1}$. If $\rho_1$ 
is faithful, we are done, if not, there is a nonzero $x_2 \in L$ such that 
$\rho_1(x_2) = 0$. Set $\rho_2 = \rho_1 \oplus \rho_{x_2}$. Note that
$$
\Ker \rho_2 = \Ker \rho_1 \cap \Ker \rho_{x_2} \subset \Ker \rho_1 ,
$$
and the inclusion is strict, since $x_2$ does not belong to the left-hand side,
but belongs to the right-hand side. Repeating this process, we get a series
of representations $\rho_1, \rho_2, \dots$, with strictly decreasing kernels.
Since $L$ is finite-dimensional, this process will terminate in a finite number
of steps on a faithful representation $\rho$. Being the direct sum of nilpotent
representations, $\rho$ is also nilpotent.
\end{proof}

\begin{lemma}[\sc Abundance of ideals of codimension $1$]\label{lem-id}
Let $I$ be a nonzero ideal of a nilpotent Lie algebra $L$. Then $I$ contains an ideal $J$ of $L$ of codimension $1$ in $I$ such that
$[L,I] \subseteq J$.
\end{lemma}

\begin{proof}
This elementary result follows almost immediately from the definition of a 
nilpotent Lie algebra by a rudimentary Jordan--H\"older-like argument. Namely,
there is a chain of ideals of $L$
$$
0 = I_0 \subset I_1 \subset \dots \subset I_{n-1} \subset I_n = L
$$
such that $\dim I_i/I_{i-1} = 1$ and $[L,I_i] \subseteq I_{i-1}$ for any 
$i=1,2,\dots,n$ (cf., e.g., \cite[\S 4, Proposition 1]{bourbaki}). The set 
$\set{i}{I \subseteq I_i}$ contains $n$ (and hence is nonempty), and does not 
contain $0$ (since $I$ is nonzero). Let $k$ be the minimal element in this set.
Then
$$
0 < \dim I/(I \cap I_{k-1}) = \dim (I + I_{k-1})/I_{k-1} \le \dim I_k/I_{k-1} 
= 1 ,
$$
and hence the ideal $I \cap I_{k-1}$ of $L$ is of codimension $1$ in $I$. We 
also have
$$
[L,I] \subseteq I \cap [L,I_k] \subseteq I \cap I_{k-1} ,
$$
as required.
\end{proof}

\begin{lemma}[\sc Factorization of linear maps]\label{lemma-lin}
Let $V$ be a vector space, $f,g$ linear maps in $\End(V)$, and 
$\Ker f \subseteq \Ker g$. Then there is a linear map $h$ in $\End(V)$ such
that $g = h \circ f$.
\end{lemma}

\begin{proof}
This is an elementary linear algebra (cf., e.g., 
\cite[Proposition 6.8]{kostr-manin}). Fix a basis $e_1, \dots, e_n$ in $V$, and
define a linear map $h: \im f \to V$ by sending $f(e_i)$ to $g(e_i)$, 
$i=1,\dots,n$. This map is well-defined, since if 
$\sum_i \lambda_i f(e_i) = \sum_i \mu_i f(e_i)$ for some 
$\lambda_i, \mu_i \in K$, then $\sum_i (\lambda_i - \mu_i) e_i$ lies in $\Ker f$
and hence in $\Ker g$, and, consequently,
$h\big(\sum_i \lambda_i f(e_i)\big) = h\big(\sum_i \mu_i f(e_i)\big)$. Extend 
$h$ to the whole $V$ arbitrarily, say, by mapping a subspace complementary to 
$\im f$ to zero.
\end{proof}

The next lemma contains the core arguments, of combinatorial character.

\begin{lemma}[\sc Distinguishing elements by representation kernels]\label{lemma-comb}
Let $L$ be a Lie algebra over a field of characteristic zero, having a faithful 
nilpotent representation. Then for any two linearly independent elements 
$x,y \in L$, there is a nilpotent representation $\rho$ of $L$ such that 
$\Ker \rho(x) \not\subset \Ker \rho(y)$. 
\end{lemma}

\begin{proof}
Suppose the contrary: there are two linearly independent elements $x,y \in L$ 
such that for any nilpotent representation $\rho$ of $L$, 
$\Ker \rho(x) \subseteq \Ker \rho(y)$. By Lemma \ref{lemma-lin}, 
\begin{equation}\label{eq-tau}
\rho(y) = h_\rho \circ \rho(x)
\end{equation} 
for some linear map $h_\rho$.

Let $\rho: L \to \End(V)$ and $\tau: L \to \End(W)$ be two nilpotent 
representations of $L$, and let $n$ and $m$ be indices of nilpotency of 
$\rho(x)$ and $\tau(x)$, respectively. The tensor product $\rho \otimes \tau$ is
also nilpotent, so writing the condition (\ref{eq-tau}) for $\rho \otimes \tau$
and for vectors 
$\rho(x)^{n-2}(v) \otimes \tau(x)^{m-1}(w)$ and
$\rho(x)^{n-1}(v) \otimes \tau(x)^{m-2}(w)$, and taking into account the same 
condition for $\rho$ and $\tau$, we get respectively
$$
h_\rho\Big(\rho(x)^{n-1}(v)\Big) \otimes \tau(x)^{m-1}(w) = 
h_{\rho \otimes \tau}\Big(\rho(x)^{n-1}(v) \otimes \tau(x)^{m-1}(w)\Big)
$$
and
$$
\rho(x)^{n-1}(v) \otimes h_\tau\Big(\tau(x)^{m-1}(w)\Big) = 
h_{\rho \otimes \tau}\Big(\rho(x)^{n-1}(v) \otimes \tau(x)^{m-1}(w)\Big)
$$
for any $v \in V$, $w \in W$. This implies that the linear maps 
$h_\rho \otimes \id$ and $\id \otimes h_\tau$ coincide on the vector space 
$\rho(x)^{n-1}(V) \otimes \tau(x)^{m-1}(W)$, whence 
\begin{equation}\label{eq-h}
h_\rho\Big(\rho(x)^{n-1}(v)\Big) = \lambda \rho(x)^{n-1}(v)
\end{equation}
and
$$
h_\tau\Big(\tau(x)^{m-1}(w)\Big) = \lambda \tau(x)^{m-1}(w)
$$
for some $\lambda \in K$. Since this holds for any pair of representations
$\rho$, $\tau$, (\ref{eq-h}) holds for any nilpotent representation $\rho$ of 
$L$ for some uniform value $\lambda \in K$.

Further, writing the condition (\ref{eq-tau}) for the tensor product 
$\rho \otimes \tau$ for vectors 
$\rho(x)^{n-3}(v) \otimes \tau(x)^{m-1}(w)$,
$\rho(x)^{n-1}(v) \otimes \tau(x)^{m-3}(w)$, and
$\rho(x)^{n-2}(v) \otimes \tau(x)^{m-2}(w)$, and taking into account 
(\ref{eq-h}), we get respectively
\begin{gather*}
h_\rho\Big(\rho(x)^{n-2}(v)\Big) \otimes \tau(x)^{m-1}(w) =
h_{\rho \otimes \tau} \Big(\rho(x)^{n-2}(v) \otimes \tau(x)^{m-1}(w)\Big) ,
\\
\rho(x)^{n-1}(v) \otimes h_\tau\Big(\tau(x)^{m-2}(w)\Big) =
h_{\rho \otimes \tau} \Big(\rho(x)^{n-1}(v) \otimes \tau(x)^{m-2}(w)\Big) ,
\end{gather*}
and
\begin{multline}\label{eq-3}
\lambda \Big(
\rho(x)^{n-1}(v) \otimes \tau(x)^{m-2}(w) + 
\rho(x)^{n-2}(v) \otimes \tau(x)^{m-1}(w)
\Big) \\ = 
h_{\rho \otimes \tau} \Big(
\rho(x)^{n-1}(v) \otimes \tau(x)^{m-2}(w) +
\rho(x)^{n-2}(v) \otimes \tau(x)^{m-1}(w) 
\Big).
\end{multline}
Summing up the first two of these equalities, and subtracting the third one, we
get that the linear maps $(h_\rho - \lambda\id) \otimes \tau(x)$ and 
$\rho(x) \otimes (h_\tau - \lambda\id)$ coincide on the vector space 
$\rho(x)^{n-2}(V) \otimes \tau(x)^{m-2}(W)$\footnote{
Added July 29, 2018: This is wrong. The linear maps in question should not 
coincide, but their sum is identically equal to zero. Accordingly, the 
subsequent reasonings are simplified leading to the same conclusion 
$h_\rho \big(\rho(x)^{n-2}(v)\big) = \lambda \rho(x)^{n-2}(v)$. See a 
forthcoming text, joint with Abdenacer Makhlouf, ``Ado theorem for nilpotent 
Hom-Lie algebras'' for a correct (and more general) reasoning.
}, 
whence $h_\rho - \lambda\id = - \mu \rho(x)$ and 
$h_\tau - \lambda\id = - \mu \tau(x)$ for some $\mu \in K$ as linear maps on
$\rho(x)^{n-2}(V)$ and $\tau(x)^{m-2}(W)$, respectively. As this holds for any
pair of nilpotent representations $\rho$, $\tau$, we get that
\begin{equation}\label{eq-mu}
h_\rho\Big(\rho(x)^{n-2}(v)\Big) = 
\lambda\rho(x)^{n-2}(v) + \mu \rho(x)^{n-1}(v)
\end{equation}
for any nilpotent representation $\rho$ of $L$, for some uniform value 
$\mu \in K$.

On the other hand, the index of nilpotency of $(\rho \otimes \rho)(x)$ is 
equal to $2n-1$, so writing (\ref{eq-mu}) for the tensor square 
$\rho \otimes \rho$, and taking into account (\ref{eq-3}) in the case 
$\tau = \rho$ (and $m=n$), we get
\begin{equation}\label{eq-binom}
\binom{2n-2}{n-1} \> \mu \> \rho(x)^{n-1}(v) \otimes \rho(x)^{n-1}(w) = 0
\end{equation}
for any $v,w \in V$, what implies $\mu = 0$, and, according to (\ref{eq-mu}),
$$
h_\rho \Big(\rho(x)^{n-2}(v)\Big) = \lambda \rho(x)^{n-2}(v)
$$
for any nilpotent representation $\rho$ of $L$. 

Repeating this procedure (considering on each step the condition (\ref{eq-tau}) 
for the tensor product of two representations $\rho$, $\tau$, and for all 
vectors of the form $\rho(x)^i(v) \otimes \tau(x)^j(w)$ with $i+j$ equal to the
index of nilpotency of $\rho \otimes \tau$ minus $(k+2)$), we consecutively
arrive at the equalities
$$
h_\rho \Big(\rho(x)^k(v)\Big) = \lambda \rho(x)^k(v)
$$
for any $1 \le k \le n$. For $k=1$\footnote{
Added July 1, 2018: Actually, $0 \le k \le n-1$, and we should take $k=0$.
}
this means that $h_\rho = \lambda\id$, and hence $\rho(y - \lambda x)$ vanishes for any nilpotent representation $\rho$ of
$L$, whence $y - \lambda x = 0$, a contradiction.
\end{proof}

Finally, we can glue all of this together:

\begin{theorem}[\sc Ado for nilpotent algebras]
A nilpotent Lie algebra over a field of characteristic zero has a faithful 
nilpotent representation.
\end{theorem}

\begin{proof}
Present a nilpotent Lie algebra $L$ as a quotient of a free nilpotent Lie 
algebra $F$ of a finite rank: $L = F/I$. We will proceed by induction on the 
dimension of $I$. The case $I=0$ is covered by Lemma \ref{lemma-free-nilp}.

Suppose that $I$ is nonzero. By Lemma \ref{lem-id}, there is an ideal $J$ of $F$
such that $J \subset I$, $\dim I/J = 1$, and $[F,I] \subseteq J$. Consequently,
$S = F/J$ is an extension of $L$ by an one-dimensional central ideal, say 
$K\tilde{z}$. 

Take an arbitrary nonzero $x \in L$ and consider its preimage $\tilde{x}$ in 
$S$. By the induction assumption, $S$ has a faithful nilpotent representation, 
and by Lemma \ref{lemma-comb}, there is a nilpotent representation 
$\rho: S \to \End(V)$ such that 
\begin{equation}\label{eq-ker}
\Ker \rho(\tilde z) \not\subset \Ker\rho( \tilde x) .
\end{equation} 
Since $\tilde z$ lies in the center of $S$, $\rho(\tilde z)$ commutes with all 
maps from $\rho(S)$, and hence the space $\Ker\rho(\tilde z)$ is an 
$S$-submodule of $V$, which also carries a natural structure of a (nilpotent) 
$L$-module, on which $x$ acts nontrivially, due to (\ref{eq-ker}). By Lemma \ref{lemma-local}, $L$ has a faithful nilpotent 
representation.
\end{proof}

\begin{remark}\label{rem-dual}
In the proof above, we may try to take a dual route, utilizing images instead of
kernels. Namely, for any representation $\rho: S \to \End(V)$, the space 
$V/\im\rho(\tilde z)$ also carries a natural structure of an $L$-module. If 
$\im \rho(\tilde x) \not\subset \im \rho(\tilde z)$, then $x$ acts on the latter
module nontrivially, and this action is nilpotent if $\rho$ is nilpotent, so we
may assume that $\im \rho(\tilde x) \subseteq \im \rho(\tilde z)$ for any nilpotent 
representation $\rho$ of $S$. By the statement dual to Lemma \ref{lemma-lin}, 
the latter condition may be rewritten as a dual one to (\ref{eq-tau}):
$$
\rho(\tilde x) = \rho(\tilde z) \circ h_\rho
$$
for some linear map $h_\rho: V \to V$. Additionally, we may be helped by the 
fact that the element $\tilde x$ in these considerations may assumed to be 
central. Indeed, write the central extension $S$ of $L$ as the vector space 
direct sum $S = L \oplus K\tilde z$, with multiplication 
$$
\{u,v\} = [u,v] + \varphi(u,v)\tilde z
$$ 
for $u,v \in L$, where $\liebrack$ is multiplication in $L$, and $\varphi$ is a
$2$-cocycle on $L$ with values in $K$. Then we may apply a ``local'' version of 
Lemma \ref{lemma-coc} for $2$-cocycles to deduce that $\varphi(x,L) = 0$, i.e. 
$\tilde x$ is central in $S$. But even with this additional input, to prove the 
appropriate dual version of Lemma \ref{lemma-comb} seems to be more tricky, and
requires an extensive consideration of associative envelopes of $\rho(S)$'s in 
the appropriate matrix algebra.
\end{remark}

\section{An alternative route}\label{sec-alt}

In the proof of the Ado theorem for nilpotent algebras in 
Section \ref{sec-proof}, the characteristic zero of the ground field $K$ is needed in two places: first,
in the proof of Lemma \ref{lem-loc-grad}, to ensure that there is a derivation 
$D$ of $tK[t]/(t^n)$ such that $\sum_{i\ge 1} x_i \otimes D(t^i) \ne 0$ (what is
wrong if the characteristic of $K$ is $p>0$, and all the exponents $i$ in 
nonzero terms of $\sum_{i\ge 1} x_i \otimes t^i$ divide $p$); and second, in the
proof of Lemma \ref{lemma-comb}, to ensure that the binomial coefficients 
arising in binomial-like formulas for powers of Lie algebra elements actions on
tensor products of representations, like those in (\ref{eq-binom}), do not
vanish.

Here we outline a possible alternative approach which should include the case of
positive characteristic. In Section \ref{sec-proof} we apply an elementary embedding
of Lemma \ref{lemma-emb} to free nilpotent Lie algebras of finite rank. A little
bit more involved argument establishes a similar embedding for a broader class 
algebras. This argument is sometimes phrased in terms of ``generic elements'' or
``generic matrices'' and is often employed in the theory of varieties of Lie, 
associative, and other kinds of algebras. Namely, a relatively free algebra $F$ 
of no more then countable rank in a variety generated by a finite-dimensional 
Lie algebra $L$ is embedded into 
\begin{equation*}
L \otimes \Big(t_1 K[t_1, \dots, t_k] + \dots + t_k K[t_1, \dots, t_k]\Big)
\end{equation*}
for some $k \in \mathbb N$ (cf., e.g., \cite[Chapter I, Lemma 5.1]{razmyslov} or
\cite[Proof of Proposition 1.3]{zaicev}). In the case of nilpotent $L$ one can
do even better and embed $F$ into $L \otimes tK[t]$ (cf. 
\cite[Lemma 1.1]{zaicev}). Also, in the case of nilpotent $L$ this embedding 
obviously factors through the ideal $L \otimes t^n K[t]$, where $n$ is the index
of nilpotency of $L$, similarly as in Lemma \ref{lemma-emb}.

(Note parenthetically that this or similar arguments are often coupled with the
Ado theorem to establish an embedding of a relatively free algebra in some 
variety of Lie algebras into an algebra with that or another finiteness condition (cf., e.g., 
\cite{zaicev}). Here we reverse this line of reasonings and use this argument to
outline a possible route to the Ado theorem). 

As the adjoint representation is nonzero on non-central elements, in view of
Lemma \ref{lemma-local}, in order to establish the Ado theorem for $F$, it will
be enough to prove that for any nonzero central element of $F$, there is a 
representation not vanishing on that element. By an easy inductive argument, one
may reduce considerations to the case where the center $\Z(L)$ of $L$ is 
one-dimensional, so to cover the characteristic $p$ case in the proof of 
Lemma \ref{lem-loc-grad}, it will be enough to consider elements of $F$ of the 
form 
$$
z \otimes (\lambda_1 t^p + \lambda_2 t^{2p} + \dots) ,
$$ 
where $z \in \Z(L)$ and $\lambda_i \in K$, which probably could be dealt with 
using additional considerations based on relative freeness of $F$.

Additionally, one may try to employ derivations of $L \otimes tK[t]/(t^n)$ of
the form other than $\id_L \otimes D$, where $D$ is a derivation of 
$tK[t]/(t^n)$. As explained in \cite[\S 3]{without-unit}, the full description 
of derivations of such current Lie algebras is, probably, a difficult task, but
one may try, for example, to employ derivations of the forms listed in \cite[Theorem 3]{without-unit}.

After establishing the Ado theorem for any relatively free algebra $F$ in a 
variety generated by a nilpotent Lie algebra $L$, we may proceed the same way
as in the proof of Theorem in Section \ref{sec-proof}, by induction on the dimension
of ideal of relations determining $L$. To establish Lemma \ref{lemma-comb} in
characteristic $p$, an additional care will be needed to deal with vanishing
of binomial coefficients occurring in the proof.

All this, however, will result in a quite long and involved proof -- at least 
much more long and involved than the existing proof of the full-fledged Ado 
theorem in positive characteristic -- so we will not pursue this approach.

\bigskip

A final remark: in all existing proofs of the Ado theorem in characteristic 
zero, the general case is derived from the nilpotent one. 
In this respect, \S \ref{sec-proof} is a good start. However, all such derivations employ the universal enveloping algebra in a non-trivial way. A short, 
``natural'', and characteristic-free proof of the Ado theorem is yet to be found.

\section*{Acknowledgements}

Thanks are due to Alexei Lebedev for useful comments on an earlier version of
this note.


\begin{thebibliography}{KM}

\bibitem[A]{ado} I.D. Ado, 
\emph{\"Uber die Darstellung von Lieschen Gruppen durch lineare Substitutionen},
Bull. Soc. Phys.-Math. Kazan \textbf{7} (1935), 3--43 (in Russian).

\bibitem[Bi]{birkhoff} G. Birkhoff, 
\emph{Representability of Lie algebras and Lie groups by matrices},
Ann. Math. \textbf{38} (1937), 526--532; reprinted in 
\emph{Selected Papers on Algebra and Topology}, Birkh\"auser, 1987, 332--338.

\bibitem[Bo]{bourbaki} N. Bourbaki, 
\emph{Groupes et Alg\`ebres de Lie. Chapitre 1}, Hermann, Paris, 1972; reprinted
by Springer, 2007.

\bibitem[KM]{kostr-manin} A.I. Kostrikin and Yu.I. Manin, 
\emph{Linear Algebra and Geometry}, 2nd ed., Nauka, 1986 (in Russian); 
revised reprint of the 1989 edition, Gordon and Breach, 1997 
(English translation).

\bibitem[LE]{lie} S. Lie and F. Engel, 
\emph{Theorie der Transformationsgruppen. Dritter and letzter Abschnitt}, 
Teubner, Leipzig, 1893; reprinted by Chelsea, N.Y., 1970.

\bibitem[N]{neretin} Yu.A. Neretin, 
\emph{A construction of finite-dimensional faithful representation of Lie algebra}, 
Proceedings of the 22nd Winter School ``Geometry and Physics'' 
(ed. J. Bure\v{s}), Rend. Circ. Mat. Palermo Suppl. \textbf{71} (2003), 159--161; \textsf{arXiv:math/0202190}.

\bibitem[R]{razmyslov} Yu.P. Razmyslov, \emph{Identities of Algebras and Their Representations}, Nauka, 1989 (in Russian); AMS, 1994 (English translation).

\bibitem[T]{tao} T. Tao, \emph{Hilbert's Fifth Problem and Related Topics},
AMS, 2014.

\bibitem[Za]{zaicev} M.V. Zaicev, \emph{Special Lie algebras}, Uspekhi Mat. Nauk \textbf{48} (1993), $\mathcal N$6, 103--140 (in Russian); Russ. Math. Surv. \textbf{48} (1993), $\mathcal N$6, 111--152 (English translation).

\bibitem[Zu]{without-unit} P. Zusmanovich, 
\emph{Invariants of Lie algebras extended over commutative algebras without 
unit}, 
J. Nonlin. Math. Phys. \textbf{17} (2010), Suppl. 1 
(special issue in memory of F.A. Berezin), 87--102; \textsf{arXiv:0901.1395}.

\end{thebibliography}
\end{document}